\documentclass[11pt,a4paper,oneside,draft,hrule]{article}
\usepackage[a4paper, left=3.2cm,right=3.2cm,top=3.2cm,bottom=3cm]{geometry}
\usepackage{fullpage}
\usepackage[OT4]{fontenc}
\usepackage{multirow}
\usepackage{amssymb}
\usepackage{amsmath}
\usepackage{amsfonts, mathrsfs}
\usepackage{graphicx}
\usepackage[cp1250]{inputenc}
\usepackage{latexsym}
\usepackage{euscript}
\usepackage{amscd}
\usepackage{exscale}
\usepackage{bbm} %{do podw?jnej jedynki}
\numberwithin{equation}{section}

\def\eq#1{(\ref{#1})}
\def\neweq#1{\begin{equation}\label{#1}}
\def\endeq{\end{equation}}
\def\RR{{\mathbb R} }

\def\di{\displaystyle}
\def\la{\lambda}
\def\ri{\rightarrow}
\def\ep{\varepsilon}
\def\vp{\varphi_1}

\def\o{\Omega }
\def\oo{\overline\Omega }

\def\mA{{\mathcal A}}
\def\mF{{\mathscr F}}

\newtheorem{them}{Theorem}[section]

\newtheorem{coro}[them]{Corollary}
\newtheorem{lema}[them]{Lemma}
\newtheorem{pro}[them]{Proposition}

\title{A biharmonic equation with singular nonlinearity}
\author{Marius Ghergu\\
School of Mathematical Sciences,\\
University College Dublin,\\ Belfield, Dublin 4, Ireland\\
E-mail: {\tt marius.ghergu@ucd.ie}}
\begin{document}

\maketitle

\medskip

\begin{abstract} We study the biharmonic equation
$\Delta^2 u =u^{-\alpha}$, $0<\alpha<1$, in a smooth and
bounded domain $\Omega\subset\RR^n$, $n\geq 2$, subject to Dirichlet boundary conditions. 
Under some suitable assumptions on $\o$ related to
the positivity of the Green function for the biharmonic operator, we prove the existence and
uniqueness of a solution.

\medskip


\medskip

\end{abstract}

\section{Introduction and the main results}
\label{intro}

\medskip

In this paper we study the biharmonic elliptic problem

\neweq{b}
\left\{\begin{aligned}
&\Delta^2 u =u^{-\alpha}\ , \; u>0 &&\mbox{ in } \Omega,\\
&u=\partial_\nu u=0&& \mbox{ on }\partial\Omega,
\end{aligned}\right.
\endeq
where $0<\alpha<1$, $\o\subset\RR^n$ ($n\geq 2$) is a smooth bounded
domain in the sense that we will describe in the following, $\nu$
is the exterior unit normal at $\partial\o$ and $\partial_\nu=\frac{\partial}{\partial\nu}$ 
is the outer normal derivative at $\partial\o$. We denote by
$G(\cdot,\cdot)$ the Green function associated with the biharmonic
operator $\Delta^2$ subject to Dirichlet boundary conditions, that is, for all
$y\in\o$, $G(\cdot, y)$ satisfies in the distributional sense:
$$
\left\{\begin{aligned}
&\Delta^2 G(\cdot, y) =\delta_y(\cdot) &&\mbox{ in } \Omega,\\
&G(\cdot, y) =\partial_\nu G(\cdot, y) =0&& \mbox{ on
}\partial\Omega.
\end{aligned}\right.
$$
The study of Green function for the biharmonic equation goes
back to Boggio \cite{bog1} in 1901. He proved that the Green
function is positive in any ball of $\RR^n$. Boggio \cite{bog2} and
Hadamard \cite{had1} conjectured that this fact should be true at
least in any smooth convex domain of $\RR^n$.

Starting with the late 1940s, various counterexamples have been
constructed that disprove the Boggio-Hadamard conjecture. For
instance, if a domain in $\RR^2$ has a right-angle, then the
associated Green function fails to be everywhere positive (see
Coffman and Duffin \cite{coff}). A similar result holds for thin
ellipses: Garabedian \cite{gara} found that in an ellipse in $\RR^2$
 with the ratio of the half axes $\simeq 2$, the Green function for the
biharmonic operator changes sign (for an elementary proof, see also
Shapiro-Tegmark \cite{shap}). In turn, if the ellipse is close to a
ball in the plane, Grunau and Sweers \cite{gsm} proved that the
Green function is positive. Recently, Grunau and Sweers \cite{gs0,
gs1, gs2}, Grunau and Robert \cite{gr} provided interesting
characterizations of the regions where the Green function is
negative. They also obtained that if a domain is sufficiently close
to the unit ball in a suitable $C^{4,\gamma}$-sense, then the biharmonic
Green function under Dirichlet boundary condition is positive.

It is worth noting here that the positivity property of the Green
function for the biharmonic operator is a special feature of the
prescribed boundary condition. Indeed, if instead of Dirichlet boundary
condition one assumes Navier boundary condition (that is, $u =
\Delta u = 0$ on $\partial\o$), then a straightforward application
of the second order comparison principle yields the positivity of
the Green function. However, even under Navier conditions there is
in general no positivity result for the Green function when the
biharmonic operator is perturbed (see, e.g., \cite{coff1, ka}).

In this paper we assume that $\o\subset \RR^n$, $n\geq 2$, is a
bounded domain that satisfies:
\begin{enumerate}
\item[$(A1)$]
the boundary $\partial\o$ is of class $C^{16}$ if $n=2$ and of class
$C^{12}$ if $n\geq 3$;
\item[$(A2)$] the Green function $G(\cdot,\cdot)$ is positive.
\end{enumerate}

The assumption $(A1)$ on the regularity of $\partial\o$ goes back to
Krasovski$\breve{{\rm \i}}$ \cite{k} and is taken from Dall'Acqua
and Sweers \cite{d} where sharp upper bounds for the Green
function are obtained. The need for condition $(A2)$ will become
more clear once we specify what it is understood by a solution of
\eq{b}. We say that $u$ is a solution of \eq{b} if
$$
u\in C(\oo),\quad u>0\quad\mbox{in }\o,$$ and $u$ satisfies the
integral equation
\neweq{intg}
u(x)=\int_\o G(x,y) u^{-\alpha}(y) dy\quad\mbox{ for all }x\in\o.
\endeq

Also remark that condition $(A2)$ above implies the standard maximum
principle for the biharmonic operator in $\o$.

Although there are several results for biharmonic equations
involving exponential or power-type nonlinearities with positive
exponents \cite{arioli,fer,gazz,gazz1}, to the best of our
knowledge, there is no such a study for \eq{b}. Our main result
concerning \eq{b} is the following.

\begin{them}\label{t1}
Assume $0<\alpha<1$ and conditions $(A1)$, $(A2)$ hold. Then, the
problem \eq{b} has a unique solution $u$ and there exist $c_1,c_2>0$
such that
\neweq{delta2}
c_1\delta^2(x)\leq u(x)\leq c_2 \delta^2(x)\quad\mbox{
in } \o,
\endeq
where $\delta(x)={\rm dist}(x,\partial\o)$.
Moreover, $u\in C^2(\oo)$ and if $0<\alpha<1/2$ then $u\in
C^3(\oo)$.
\end{them}

The existence of a solution will be obtained by means of 
Schauder fixed point theorem. To this aim, we employ the sharp
estimates for Green function given in \cite{d}. The uniqueness
relies heavily on the boundary estimate \eq{delta2} which is obtained 
by using the behavior of the Green function (see Proposition \ref{p2} below).

The remaining part of the paper is organized as follows. In Section
2 we derive some preliminary results concerning \eq{b}. Section 3 is
devoted to the proof of Theorem \ref{t1}.

\newpage

\section{Preliminary results}

In this section we colect some useful results regarding problem
\eq{b}. The first result in this sense is due to Dall'Acqua and
Sweers \cite[Theorem 12, Lemma C.2]{d} and provides upper bounds for
the Green function of the biharmonic operator subject to Dirichlet
boundary conditions.

\begin{pro}\label{p1} {\rm (see \cite{d}) }
Let $k$ be a $n-$dimensional multi-index. Then, there exists a
positive constant $c$ depending on $\o$ and $k$ such that for any
$x,y\in\o$ we have

\begin{enumerate}
\item[(i)] For $|k|\geq 2:$
\begin{enumerate}
\item[(i1)] if $n>4-|k|$ then
$$
|D_x^k G(x,y)|\leq c|x-y|^{4-n-|k|}\min\left\{1,\frac{\delta(y)}{|x-y|} \right\}^2,
$$
\item[(i2)] if $n=4-|k|$ then
$$
|D_x^k G(x,y)|\leq c\log \left (2+\frac{\delta(y)}{|x-y|}
\right)\min\left\{1,\frac{\delta(y)}{|x-y|} \right\}^2,
$$
\item[(i3)] if $n<4-|k|$ then
$$
|D_x^k G(x,y)|\leq
c\delta(y)^{4-n-|k|}\min\left\{1,\frac{\delta(y)}{|x-y|}
\right\}^{n+|k|-2}.
$$
\end{enumerate}
\item[(ii)]  For $|k|< 2:$
\begin{enumerate}
\item[(ii1)] if $n>4-|k|$ then
$$
|D_x^k G(x,y)|\leq
c|x-y|^{4-n-|k|}\min\left\{1,\frac{\delta(x)}{|x-y|}
\right\}^{2-|k|} \min\left\{1,\frac{\delta(y)}{|x-y|} \right\}^{2},
$$
\item[(ii2)] if $n=4-|k|$ then
$$
|D_x^k G(x,y)|\leq c\log \left (2+\frac{\delta(y)}{|x-y|}
\right)\min\left\{1,\frac{\delta(x)}{|x-y|}
\right\}^{2-|k|}\min\left\{1,\frac{\delta(y)}{|x-y|} \right\}^2,
$$
\item[(ii3)] if $2(2-|k|)\leq n<4-|k|$ then
$$
|D_x^k G(x,y)|\leq
c\delta(y)^{4-n-|k|}\min\left\{1,\frac{\delta(x)}{|x-y|}
\right\}^{2-|k|} \min\left\{1,\frac{\delta(y)}{|x-y|}
\right\}^{n+|k|-2},
$$

\item[(ii4)] if $n<2(2-|k|)$ then
$$
|D_x^k G(x,y)|\leq
c\delta^{2-|k|-n/2}(x)\delta^{2-n/2}(y)\min\left\{1,\frac{\delta(x)}{|x-y|}
\right\}^{n/2} \min\left\{1,\frac{\delta(y)}{|x-y|} \right\}^{n/2}.
$$

\end{enumerate}
\end{enumerate}
\end{pro}

Let $\vp$ be the first eigenfunction of $(-\Delta)$ in $H_0^1(\o)$. It is well known that $\vp$ has constant sign in $\o$, so by a suitable normalization we may assume $\vp>0$ in $\o$. Therefore, $\vp$ satisfies
\neweq{ff}
\left\{\begin{aligned}
&-\Delta \vp =\la_1 \vp\ , \; \vp>0 &&\mbox{ in } \Omega,\\
&\vp=0&& \mbox{ on }\partial\Omega,
\end{aligned}\right.
\endeq
where $\la_1>0$ is the first eigenvalue of $(-\Delta)$. By Hopf
maximum principle \cite{pr} we have $\partial_\nu \vp<0$ on
$\partial\o$. Also, by the regularity of $\o$ we have $\vp\in C^4(\oo)$ and
\neweq{fi}
c\delta(x)\leq \vp(x)\leq \frac{1}{c}\delta(x)\quad\mbox{ in }\o,
\endeq
for some $0<c<1$.

\begin{pro}\label{p2}
Let $u$ be a solution of problem \eq{b}. Then, there exist
$c_1,c_2>0$ such that $u$ satisfies \eq{delta2}.
\end{pro}
\begin{proof}
Let $a(x)=\vp^2(x)$, $x\in\oo$. It is easy to see that since $\vp\in C^4(\oo)$ then
$$
f:=\Delta^2 a=2\la_1^2\vp^2+\sum_{i=1}^n\left[8\frac{\partial
\vp}{\partial x_i}\frac{\partial^3 \vp}{\partial
x_i^3}+6\left(\frac{\partial^2 \vp}{\partial x_i^2}\right)^2 \right]
$$
is bounded in $\oo$, so, by the continuity of $u$ there exists $m>0$
small enough such that
$$
u(x)-ma(x)=\int_\o G(x,y) \Big[u^{-\alpha}(y)-m f(y)\Big] dy\geq
0\quad\mbox{ for all }x\in\o.
$$
Therefore, 
\neweq{e1}
u(x)\geq ma(x)\geq c_0\delta^2(x)\quad\mbox{ in }\o,
\endeq
for some $c_0>0$. This proves the first part of the inequality in
\eq{delta2}. For the second part, assume first $n>4$ and let $x\in
\o$. Using Proposition \ref{p1}(ii1), for all $y\in\o$  we have
\neweq{green}
\begin{aligned}
G(x,y)&\leq c|x-y|^{2-n}\delta^2(x)\min\left\{1,\frac{\delta(y)}{|x-y|} \right\}^{2}\\
&\leq c|x-y|^{2-n}\delta^{2}(x)\min\left\{1,\frac{\delta(y)}{|x-y|} \right\}^{2\alpha}\\
&=c |x-y|^{2-2\alpha-n}\delta^2(x)\delta^{2\alpha}(y).
\end{aligned}
\endeq
Now, from \eq{e1} and \eq{green} we have
\neweq{ww}
\begin{aligned}
u(x)&=\int_\o G(x,y) u^{-\alpha}(y) dy\\
&\leq c_1 \int_\o G(x,y)\delta^{-2\alpha}(y)dy\
\leq c_2\delta^2(x)\int_\o |x-y|^{2-2\alpha-n}dy\\
&\leq c_2\delta^2(x)\int_{0\leq |x-y|\leq {\rm diam}(\o)} |x-y|^{2-2\alpha-n}dy\\
&=c_2\delta^2(x)\int_0^{{\rm diam}(\o)}t^{1-2\alpha}dt\\
&\leq c_3\delta^2(x).
\end{aligned}
\endeq

Assume now $n=4$. We use Proposition \ref{p1}(ii2) to derive a
similar inequality to \eq{green}. More precisely, for all $y\in\o$
we have
\neweq{w1}
\begin{aligned}
G(x,y)&\leq c\log\left(2+\frac{\delta(y)}{|x-y|}\right) \min\left\{1,\frac{\delta(x)}{|x-y|} \right\}^{2}
\min\left\{1,\frac{\delta(y)}{|x-y|} \right\}^{2\alpha}\\
&\leq
c|x-y|^{-2-2\alpha}\delta^2(x)\delta^{2\alpha}(y)\log\left(2+\frac{{\rm
diam}(\o)}{|x-y|}\right).
\end{aligned}
\endeq
If $n=3$, let $\beta=\max\{0,2\alpha-1/2\}<3/2$ and by Proposition
\ref{p1}(ii4) we have
\neweq{w2}
\begin{aligned}
G(x,y)&\leq c\delta^{1/2}(x) \delta^{1/2}(y)
\min\left\{1,\frac{\delta(x)}{|x-y|} \right\}^{3/2}
\min\left\{1,\frac{\delta(y)}{|x-y|} \right\}^{3/2}\\
&\leq c|x-y|^{-3/2-\beta}\delta^2(x)\delta^{\beta+1/2}(y)\\
&\leq C|x-y|^{-3/2-\beta}\delta^2(x)\delta^{2\alpha}(y).
\end{aligned}
\endeq
Finally, if $n=2$, let $\beta=\max\{0,2\alpha-1\}<1$ and by
Proposition \ref{p1}(ii3) we have
\neweq{w3}
\begin{aligned}
G(x,y)&\leq c\delta(x) \delta(y)
\min\left\{1,\frac{\delta(x)}{|x-y|} \right\}
\min\left\{1,\frac{\delta(y)}{|x-y|} \right\}\\
&\leq c|x-y|^{-1}\delta^2(x) \delta(y)
\min\left\{1,\frac{\delta(y)}{|x-y|} \right\}^{\beta}\\
&\leq c|x-y|^{-1-\beta}\delta^2(x) \delta^{1+\beta}(y)\\
&\leq C|x-y|^{-1-\beta}\delta^2(x) \delta^{2\alpha}(y).
\end{aligned}
\endeq
We now use the estimates \eq{w1}-\eq{w3} to derive a similar
inequality to that in \eq{ww}.

This completes the proof of Proposition \ref{p2}.
\end{proof}

\begin{pro}\label{p3}
Let $0<\alpha<1$ and $u\in C(\oo)$ be such that $u(x)\geq c\delta^2(x)$
in $\o$ for some $c>0$. Consider
\[
w(x)=\int_\o G(x,y) u^{-\alpha}(y) dy \quad\mbox{ for all }x\in\oo.
\]
Then
\begin{enumerate}
\item[(i)] $w\in C^2(\oo)$;
\item[(ii)] $w \in C^3(\oo)$ for any $0<\alpha<1/2$.
\end{enumerate}
\end{pro}
\begin{proof} With the same proof as in Proposition \ref{p2} it is
easy to see that $v$ is well defined. For $0<\ep<1$ small, define
$\o_\ep=\{x\in\oo:\delta(x)<\ep\}$. Set
$$
u_\ep(x)=\left\{
\begin{aligned}
&u(x)&\quad\mbox{ if } x\in\o\setminus \o_\ep,\\
&0& \quad\mbox{ if } x\in\o_\ep.
\end{aligned}
\right.$$ and
$$
w_\ep(x)=\int_\o G(x,y) u_\ep^{-\alpha}(y) dy \quad\mbox{ for all
}x\in\oo.
$$
Since $u_\ep^{-\alpha}$ is bounded in $\oo$, by the estimates in
Proposition \ref{p1} it follows that $w_\ep\in C^3(\oo)$ and
$$
D_x^kw_\ep(x)=\int_\o D_x^k G(x,y) u_\ep^{-\alpha}(y) dy \quad\mbox{
for all }x\in\oo,
$$
for any $n-$dimensional multi-index $k$ with $|k|\leq 3$. 
The proof of this fact is similar to that of Lemma 4.1 in \cite{gt}.
We employ in the following the same approach as in \cite{gt} to show that $w\in C^2(\oo)$ 
(resp. $w\in C^3(\oo)$ if $0<\alpha<1/2$).

Assume first that $n> 4$ and let $k$ be a $n-$dimensional multi-index with $|k|\leq
2$. Fix $\beta>0$ such that $2\alpha<\beta<2$.

By Proposition \ref{p1}(i1) (if $|k|=2$) and (ii1) (if $|k|\leq 1$)
we have
$$
\begin{aligned}
\left|D_x^k w_\ep(x)-\int_\o D_x^k G(x,y) u^{-\alpha}(y)dy\right|&\leq
\int_{\o_\ep} |D_x^k G(x,y)|u^{-\alpha}(y) dy\\
&\leq c_1 \int_{\o_\ep}
|x-y|^{4-|k|-n}\delta^{-2\alpha}(y)\min\left\{1,\frac{\delta(y)}{|x-y|}
\right\}^{2}dy\\
&\leq c_1 \int_{\o_\ep}
|x-y|^{4-|k|-n}\delta^{-2\alpha}(y)\min\left\{1,\frac{\delta(y)}{|x-y|}
\right\}^{\beta}dy\\
&\leq c_1 \int_{\o_\ep}
|x-y|^{4-|k|-\beta-n}\delta^{\beta-2\alpha}(y)dy\\
&\leq c_1 \ep^{\beta-2\alpha}\int_{\o}
|x-y|^{4-|k|-\beta-n}dy\\
&\leq c_1 \ep^{\beta-2\alpha}\int_{0\leq |x-y|\leq {\rm diam}(\o)}
|x-y|^{4-|k|-\beta-n}dy\\
&\leq c_1 \ep^{\beta-2\alpha}\int_{0}^{{\rm diam}(\o)}
t^{3-|k|-\beta}dt\\
&\leq c_2 \ep^{\beta-2\alpha}\int_{0}^{{\rm diam}(\o)}
t^{1-\beta}dt\leq c_3 \ep^{\beta-2\alpha}\ri 0\quad \mbox{ as
}\ep\ri 0.
\end{aligned}
$$

The case $2\leq n\leq 4$ can be analyzed in the same way. For
instance, if $n=3$ and $|k|=1$, we use Proposition \ref{p1}(ii2) to
derive
$$
\begin{aligned}
\left|D_x^k w_\ep(x)-\int_\o D_x^k G(x,y) u^{-\alpha}(y)dy\right|&\leq
c_1 \int_{\o_\ep} \log\left(2+\frac{\delta(y)}{|x-y|}\right)
\delta^{-2\alpha}(y)\min\left\{1,\frac{\delta(y)}{|x-y|}
\right\}^{2}dy\\
&\leq c_1 \int_{\o_\ep}
|x-y|^{-\beta}\log\left(2+\frac{\delta(y)}{|x-y|}\right)
\delta^{\beta-2\alpha}(y)dy\\
&\leq c_1 \ep^{\beta-2\alpha} \int_{\o_\ep}
|x-y|^{-\beta}\log\left(2+\frac{{\rm diam}(\o)}{|x-y|}\right)dy\\
&\leq c_2 \ep^{\beta-2\alpha}\int_{0}^{{\rm
diam}(\o)}t^{2-\beta}\log\left(2+\frac{{\rm diam}(\o)}{t}\right)dt\\
&\leq c_3 \ep^{\beta-2\alpha}\ri 0\quad \mbox{ as }\ep\ri 0.
\end{aligned}
$$

We have obtained that
$$
D_x^kw_\ep\ri \int_\o D_x^kG(\cdot, y)u^{-\alpha}(y) dy \quad\mbox{
uniformly as } \ep\ri 0,$$ for any $n-$dimensional multi-index $k$ with $0\leq
|k|\leq 2$. It follows that $w\in C^2(\oo)$ and
$$
D_x^kw(x)= \int_\o D_x^kG(x, y)u^{-\alpha}(y) dy \quad\mbox{ for all
} x\in\oo,$$ for any multi-index $k$ with $0\leq |k|\leq 2$.

(ii) Let $k$ be a multi-index with $|k|=3$ and $2\alpha<\beta<1$.
From Proposition \ref{p1}(i1) we have
$$
\begin{aligned}
\left|D_x^k w_\ep(x)-\int_\o D_x^k G(x,y) u^{-\alpha}(y)dy\right|&\leq
\int_{\o_\ep} |D_x^k G(x,y)|u^{-\alpha}(y) dy\\
&\leq c_1 \int_{\o_\ep}
|x-y|^{1-n}\delta^{-2\alpha}(y)\min\left\{1,\frac{\delta(y)}{|x-y|}
\right\}^{\beta}dy\\
&\leq c_1 \ep^{\beta-2\alpha}\int_{\o}
|x-y|^{1-n-\beta}dy\\
&\leq c_1 \ep^{\beta-2\alpha}\int_{0}^{{\rm diam}(\o)}
t^{-\beta}dt\leq c_2 \ep^{\beta-2\alpha}\ri 0\quad \mbox{ as }\ep\ri
0,
\end{aligned}
$$
since $\beta<1$. With the same arguments as above we find $w\in
C^3(\oo)$. This completes the proof.
\end{proof}

\section{Proof of Theorem \ref{t1}}

Let $a(x)=\vp^2(x)$, $x\in\oo$. Motivated by Proposition \ref{p2}
we will be looking for solutions $u$ of \eq{b} in the form $$u(x)=a(x)v(x)$$
where $v\in C(\oo)$, $v>0$ in $\oo$. This leads us to the following integral equation for $v$:
\neweq{v1}
v(x)=\frac{1}{a(x)}\int_\o \frac{G(x,y)}{a^{\alpha}(y)} v^{-\alpha}(y)dy\quad\mbox{ for all }x\in\oo.
\endeq
We can now regard \eq{v1} as the fixed point problem
\[
\mF(v)=v,
\]
where
\[
\mF(v)= \frac{1}{a(x)}\int_\o \frac{G(x,y)}{a^{\alpha}(y)} v^{-\alpha}(y)dy.
\]
Remark that $\mF$ is an integral operator of the form
\[
\mF(v)=\int_\o K(x,y) v^{-\alpha}(y)dy,
\]
where the kernel $K$ is given by
\[
K:\oo\times\o\ri [0,\infty],\quad
K(x,y)=\left\{
\begin{aligned}
&\frac{G(x,y)}{a(x)a^{\alpha}(y)}&&\quad\mbox{ if }x,y\in\o,\\
&\frac{\partial^2_\nu G(x,y)}{\partial^2_\nu a(x) a^\alpha(y)} &&\quad\mbox{ if }x\in\partial\o, y\in\o.
\end{aligned}
\right.
\]
Note that $K$ is well defined since $\partial^2_\nu a(x)=2(\partial_\nu \vp(x))^2>0$ on $\partial\o$.

We first need the following result.

\begin{lema}\label{l1} (i) For any $y\in\o$, the function $K(\cdot,y):\oo\ri [0,\infty]$ is continuous;

(ii) The mapping
\[
\oo\ni x\mapsto \int_\o K(x,y)dy
\]
is continuous and there exists $M>1$ such that
\neweq{m}
\frac{1}{M}\leq \int_\o K(x,y)dy\leq M\quad\mbox{ for all } x\in\oo.
\endeq
\end{lema}
\begin{proof}
Since the Green function is continuous on $\o\times \o$, it remains to prove the continuity of
$K(\cdot,y)$ on $\partial\o$. Let $\ep>0$. Since $G(\cdot, y)\in C^4(\oo\setminus\{y\})$ and $a\in C^4(\oo)$, for any
$z\in\partial\o$ we have
$$
\begin{aligned}
G(z+t\nu,y)=&\, t^2\left(\frac{1}{2}\partial^2_\nu G(z,y)+G_1(z,t)\right)\quad\mbox{ as } t\nearrow 0,\\
a(z+t\nu,y)=&\, t^2\left(\frac{1}{2}\partial^2_\nu a(z,y)+a_1(z,t)\right)\quad\mbox{ as } t\nearrow 0,
\end{aligned}
$$
where
$$
\lim_{t\nearrow 0}G_1(z,t)=\lim_{t\nearrow 0}a_1(z,t)=0\quad\mbox{ uniformly for } z\in\partial\o.
$$

Hence, as $t\nearrow 0$ we have
$$
\begin{aligned}
|K(z+t\nu,y)-K(z,y)|&=\left |\frac{\frac{1}{2}\partial^2_\nu G(z,y)+G_1(z,t)}{\frac{1}{2}\partial^2_\nu
a(z,y)+a_1(z,t)}-\frac{\partial^2_\nu G(z,y)}{\partial^2_\nu a(z,y)} \right |\\
&\leq \frac{|G_1(z,y)|\partial^2_\nu a(z,y)+|a_1(z,t)||\partial^2_\nu G(z,y)|}{\partial^2_\nu a(z,y)|\frac{1}{2}\partial^2_\nu
a(z,y)+a_1(z,t)|}.
\end{aligned}
$$
Thus, there exists $\eta_1>0$ such that
\neweq{k1}
|K(z+t\nu,y)-K(z,y)|<\frac{\ep}{2}\quad\mbox{ for all } z\in\partial\o \mbox{ and }-\eta_1<t<0 .
\endeq
Also, by the smoothness of the boundary $\partial\o$ there exists $\eta_2>0$ such that
\neweq{k2}
|K(z,y)-K(\bar z,y)|<\frac{\ep}{2}\quad\mbox{ for all } z,\bar z\in\partial\o, |z-\bar z|<\eta_2.
\endeq
Define $\eta=\min\{\eta_1,\eta_2\}/2$ and fix $z\in\partial\o$. Let now $x\in\oo$ be such that $|x-z|<\eta$. Also,
let $\bar x\in\partial\o$ be such that $|x-\bar x|=\delta(x)={\rm dist}(x,\partial\o)$. Then $|x-\bar x|\leq |x-z|<\eta$ and
$|\bar x-z|\leq |x-\bar x|+|z-x|<2\eta<\eta_2$ so by \eq{k2} we have
\neweq{k3}
|K(\bar x,y)-K(z,y)|<\frac{\ep}{2}.
\endeq
Now, from \eq{k1} and \eq{k3} we obtain
$$
|K(x,y)-K(z,y)|\leq |K(x,y)-K(\bar x,y)|+|K(\bar x,y)-K(z,y)|<\ep
$$
so $K(\cdot,y)$ is continuous at $z\in\partial\o$. This completes the proof of (i).

(ii) Assume first $n>4$. Using \eq{fi} and Proposition \ref{p1}(ii1) we
have
$$
\begin{aligned}
K(x,y) &\leq c_1\delta^{-2}(x)\delta^{-2\alpha}(y)G(x,y)\\
& \leq c_2 |x-y|^{2-n}\delta^{-2\alpha}(y)\min\left\{1,\frac{\delta(y)}{|x-y|} \right\}^{2}\\
&\leq c_2 |x-y|^{2-n}\delta^{-2\alpha}(y)\min\left\{1,\frac{\delta(y)}{|x-y|} \right\}^{2\alpha}\\
&\leq c_2 |x-y|^{2-2\alpha-n} \quad\mbox{ for all }x,y\in\o.
\end{aligned}
$$
Since $0<\alpha<1$, the mapping $x\mapsto |x-y|^{2-2\alpha-n}$ is
integrable on $\o$, so by means of Lebesgue's dominated convergence
Theorem we deduce that $\oo\ni x \mapsto \int_{\o}K(x,y)dy$ is
continuous. This fact combined with $K>0$ in $\o$ proves the
existence of a number $M>1$ that satisfies \eq{m}.

For $2\leq n\leq 4$ we proceed similarly with different estimates (as in the proof of Proposition \ref{p2}) to derive the same conclusion.
\end{proof}

Let $M>1$ satisfy \eq{m} and fix $0<\ep<1$ such that
\neweq{ee}
\ep^{1-\alpha^2}\leq M^{-1-\alpha}.
\endeq
Define
$$
g_\ep:\RR\ri \RR,\quad
g_\ep(t)=\left\{
\begin{aligned}
\ep^{-\alpha}&&\quad\mbox{ if }t<\ep,\\
t^{-\alpha}&&\quad\mbox{ if }t\geq \ep,
\end{aligned}
\right.
$$
and for any $v\in C(\oo)$, $v>0$ in $\oo$ consider the operator
\[
T_\ep(v)(x)=\int_\o K(x,y)g_\ep(v(y))dy\quad\mbox{ for all }x\in\oo.
\]
If $v\in C(\oo)$ satisfies $v>0$ in $\oo$, then $g_\ep(v)\leq \ep^{-\alpha}$ in $\oo$ so by \eq{m} we find
$T_\ep(v)\leq M\ep^{-\alpha}$ in $\oo$. Let now
\[
v_1\equiv M^{-1-\alpha}\ep^{\alpha^2}, \quad  v_2\equiv M\ep^{-\alpha}.
\]
and
\[
[v_1,v_2]=\{ v\in C(\oo): \; v_1\leq v\leq v_2\}.
\]
By Lemma \eq{m} it is easy to see that $T_\ep([v_1,v_2])\subseteq [v_1,v_2]$. Further, by Lemma \ref{l1} and Arzela-Ascoli theorem, it follows that
\[
T_\ep:[v_1,v_2]\ri  [v_1,v_2] \quad\mbox{ is compact}.
\]
Hence, by Schauder fixed point theorem, there exists $v\in C(\oo)$,
$v_1\leq v\leq v_2$ in $\oo$ such that $T_\ep(v)=v$. By \eq{ee} it
follows that $v\geq v_1\geq \ep$ in $\oo$, so $g_\ep(v)=v^{-\alpha}$.
Therefore, $v$ satisfies \eq{v1}, that is, $u=av$ is a solution of
\eq{b}. Now, the the boundary estimate \eq{delta2} and the
regularity of solution $u$ follows from Proposition \ref{p2} and
Proposition \ref{p3} respectively. In the following we derive the
uniqueness of the solution to \eq{b}.

Let $u_1$, $u_2$ be two solutions of \eq{b}. Using Proposition
\ref{p2} there exists $0<c<1$ such that
\neweq{ui}
c\delta^2(x)\leq u_i(x)\leq \frac{1}{c}\delta^2(x)\quad\mbox{  in
}\o,\; i=1,2.
\endeq
This means that we can find a constant $C>1$ such that $Cu_1\geq
u_2$ and $Cu_2\geq u_1$ in $\o$.

We claim that $u_1\geq u_2$ in $\o$. Supposing the contrary, let
$$
M=\inf\{A>1:Au_1\geq u_2 \;\mbox{ in }\o\}.
$$
By our assumption, we have $M>1$. From $Mu_1\geq u_2$ in $\o$, it
follows that
$$
M^\alpha u_2(x)-u_1(x)=\int_\o G(x,y) \Big[M^\alpha
u_2^{-\alpha}(y)-u_1^{-\alpha}(y)\Big] dy\geq 0\quad\mbox{ for all
}x\in\o,
$$
and then
$$
M^{\alpha^2} u_1(x)-u_2(x)=\int_\o G(x,y) \Big[M^{\alpha^2}
u_1^{-\alpha}(y)-u_2^{-\alpha}(y)\Big] dy\geq 0\quad\mbox{ for all
}x\in\o.
$$

We have thus obtained $M^{\alpha^2} u_1\geq u_2$ in $\o$. Since
$M>1$ and $\alpha^2<1$, this last inequality contradicts the
minimality of $M$. Hence, $u_1\geq u_2$ in $\o$. Similarly we deduce
$u_1\leq u_2$ in $\o$, so $u_1\equiv u_2$ and the uniqueness is
proved. This finishes the proof of Theorem \ref{t1}.

\end{document}